\documentclass[reqno]{amsart}
\usepackage{epsfig,epstopdf,multicol,amsfonts,amsthm,amssymb,latexsym,amsmath,enumerate,hyperref,color}
\newtheorem{theorem}{Theorem}

\newtheorem{lemma}[theorem]{Lemma}

\newtheorem{remark}{Remark}

\title[Solution form of a higher order system of difference equation $\ldots$]{Solution Form of a higher order system of difference equation and dynamical behavior of its special case}

\author[N. Haddad, N. Touafek, J. F. T. Rabago]{Nabila Haddad, Nouressadat Touafek and Julius Fergy Rabago}
\address{Nabila Haddad, LMAM Laboratory, Mathematics Department, Jijel University,
Jijel 18000, Algeria} \email{nabilahaddadt@yahoo.com}
\address{Nouressadat Touafek, LMAM Laboratory, Mathematics Department, Jijel University,
Jijel 18000, Algeria} \email{touafek@univ-jijel.dz}
\address{Julius Fergy T. Rabago, Department of Mathematics and Computer Science, College of Science, University of the Philippines Baguio,
Baguio City 2600, Philippines} \email{jfrabago@gmail.com}


\date{\today}
\begin{document}
\maketitle
\begin{abstract}
The solution form of the system of nonlinear difference equations
\begin{equation*}
x_{n+1} = \frac{x_{n-k+1}^{p}y_{n}}{a y_{n-k}^{p}+b y_{n}},\
y_{n+1} = \frac{y_{n-k+1}^{p}x_{n}}{\alpha x_{n-k}^{p}+\beta x_{n}},
\quad n, p \in \mathbb{N}_{0},\ k\in \mathbb{N},
\end{equation*}
where the coefficients
$a, b, \alpha, \beta$ and the initial values $x_{-i},y_{-i},i\in\{0,1,\ldots,k\}$ are real numbers, are obtained through elementary methods.
In addition, the existence of periodic solution of the above equation for the case $p=1$ are determined.
Numerical examples are presented to illustrate the results exhibited in the paper.
\end{abstract}

\section{Introduction}

Lately, nonlinear difference equations have attracted the attentions of many mathematicians. 
In fact, we have witnessed a rapid growth of interest in these types of equations in the previous decade.
Perhaps, the motivations rooted from the fact that these type of equations have numerous applications not only in the fields of mathematics but also in related sciences, 
especially in discrete system biology, epidemiology, ecology, economics, physics and so on. 
We believe that this line of research will continue to attract the attentions of more researchers in future years as more fascinating and intriguing results are obtained and delivered in recent investigations.
Quite recently, the problem of finding the closed-form solutions of nonlinear difference equations have become a trend over this research topic. 
As a matter of fact, numerous papers dealt with the problem of solving nonlinear difference equations in any way possible, see, for instance \cite{berg}--\cite{yaz1}. 
Apparently, finding the solution form of these type of equations is, in general, a very challenging task. 
Nevertheless, various method were offered recently to reduce complicated nonlinear difference equations into linear forms which have already known solution forms. 
For instance, through transforming  into linear types, a large classes of nonlinear difference equations were resolved in closed-forms (see, e.g., \cite{stev1}--\cite{stev18}).

In an earlier paper, Touafek and Elsayed established in \cite{toua1} the solution form of the system of nonlinear difference equations
\begin{equation}
\label{O}
x_{n+1}=\frac{x_{n-1}y_{n}}{y_{n-2}\pm y_{n}},
\quad 
y_{n+1}=\frac{y_{n-1}x_{n}}{x_{n-2}\pm x_{n}},
\qquad n \in\mathbb{N}_{0},
\end{equation}
with real nonzero initial values $x_{-2}$, $x_{-1}$, $x_0$, $y_{-2}$, $y_{-1}$ and $y_0$.

In this existing work, we shall generalize the results exhibited in \cite{toua1} through examining the solution form of the following system of difference equations
\begin{equation} 
\label{T}
x_{n+1} = \frac{x_{n-k+1}^{p}y_{n}}{a y_{n-k}^{p}+b y_{n}},
\quad 
y_{n+1} = \frac{y_{n-k+1}^{p}x_{n}}{\alpha x_{n-k}^{p}+\beta x_{n}},
\qquad n,\ p  \in \mathbb{N}_{0},\ k\in \mathbb{N}.
\end{equation}
Here, the parameters $a$, $b$, $\alpha$, $\beta$ and initial values
$\{x_i\}_{i=-k}^0$, $\{y_i\}_{i=-k}^0$ are all real numbers.
In the case when $p=1$, we give the necessary and sufficient conditions for equation \eqref{T} to have a periodic solution of period $k$ (not necessarily a prime period).
We remark that, by a well defined solution of system \eqref{T}, we mean a solution such that $\left(a y_{n-k}^{p}+b y_{n}\right)\left(\alpha x_{n-k}^{p}+\beta x_{n}\right)\neq0$, for all values of $n \in\mathbb{N}_{0}$.

In deriving the form of solution of the two equations, the following lemma shall be useful.
\begin{lemma}[cf. \cite{sab}] \label{lem1} 
Consider the linear difference equation
\begin{equation*}
y_{n+2}=ay_{n}+b,\qquad n\in \mathbb{N}_{0}. \label{1.3}
\end{equation*}
Then, 
\[
\forall n \in \mathbb{N}_0 : y_{2n + i}= \left\{
\begin{array}{ll}
y_i+b n, & a= 1,\\[1em]
a^{n}y_i+\left(\displaystyle\frac{a^{n}-1}{a-1}\right)b, & \text{otherwise},
\end{array}
\right.\text{for}\ i=0,1.
\]
\end{lemma}
In the sequel, as usual, we assume that $\prod_{j=i}^{k}A_{j}=1$ and $\sum_{j=i}^{k}A_{j}=0$, for all $k<i$.

Now, we turn on the organization of the paper.
In the next section (Section 2), we shall derive analytically, reducing the system to linear types and then utilizing Lemma \ref{lem1}, the form of solutions of system \eqref{T}.
In Section 3, we examine the dynamics of the system for the case $p=1$.
Particularly, we examine the boundedness, the asymptotic behavior and periodicity of solutions of the system \eqref{T} with $p=1$, 
Finally, we end our paper by providing some examples to illustrate numerically our results in Section 3. 

\section{Form of Solutions of system \eqref{T}}
In this section, we determine the solution form of the system of difference equations \eqref{T}.
Throughout the discussion, we assume, without further mentioning, that the solutions of system \eqref{T} being studied are well-defined.

Now, to begin with, we rearrange system \eqref{T} as follows
\[
\frac{x_{n-k+1}^{p}}{x_{n+1}}=a \frac{y_{n-k}^{p}}{y_{n}}+b,\qquad
\frac{y_{n-k+1}^{p}}{y_{n+1}}=\alpha\frac{x_{n-k}^{p}}{x_{n}}+\beta.
\]
Putting
\begin{equation}\label{e2}
	u_{n}=\frac{x_{n-k}^{p}}{x_{n}},
	\qquad v_{n}=\frac{y_{n-k}^{p}}{y_{n}},
	\qquad \forall n\in\mathbb{N}_{0},
\end{equation}
we get
\[
	u_{n+1}=av_{n}+b,
	\qquad v_{n+1}=\alpha u_{n}+\beta,
	\qquad \forall n\in\mathbb{N}_{0}.
\]
So
\[
u_{n+2}=a\alpha u_{n}+a\beta+b,
\qquad v_{n+2}=a\alpha v_{n}+\alpha b+\beta,
\qquad \forall n\in \mathbb{N}_{0}.
\]
From this, we get, for all $n\in\mathbb{N}_0$, the following linear first order nonhomogeneous difference equations,
\begin{align*}
	u_{2(n+1)} &= a\alpha u_{2n}+a\beta+b,		\qquad u_{2(n+1)+1}=a\alpha u_{2n+1}+a\beta+b,\\
	v_{2(n+1)} &=a\alpha v_{2n}+\alpha b+\beta,	\qquad v_{2(n+1)+1}=a\alpha v_{2n+1}+\alpha b+\beta.
\end{align*}
Then, in view of Lemma \ref{lem1}, we get
\begin{equation}\label{f1}
\forall n\in\mathbb{N}_0 : \ u_{2n + i} 
	= \left\{
	\begin{array}{ll}
		u_{i}+\left(a\beta+b\right) n, & a\alpha= 1,\\[1em]
		\left(a\alpha\right)^{n}u_{i}+\left(\displaystyle\frac{\left(a\alpha\right)^{n}-1}{\left(a\alpha\right)-1}\right)\left(a\beta+b\right), & a\alpha \neq 1.
	\end{array}
	\right.,\ i = 0,1
\end{equation}
and
\begin{equation}\label{f32}
\forall n\in\mathbb{N}_0 : \  v_{2n + i} 
	= \left\{
	\begin{array}{ll}
		v_{i}+\left(\alpha b+\beta\right) n, & a\alpha= 1,\\[1em]
		\left(a\alpha\right)^{n}v_{i}+\left(\displaystyle\frac{\left(a\alpha\right)^{n}-1}{\left(a\alpha\right)-1}\right)\left(a\beta+b\right), & a\alpha \neq 1.
	\end{array}
	\right.,\ i = 0,1.
\end{equation}
From equations \eqref{e2}-\eqref{f32}, it follows that for all $n\in\mathbb{N}_0$, we have
\begin{align}
u_{2n}&
	= \left\{
		\begin{array}{ll}
		\displaystyle\frac{x_{-k}^{p}+\left(a\beta+b\right)n x_{0}}{x_{0}}, & a\alpha= 1,\\[1.5em]
		\displaystyle\frac{\left(a\alpha\right)^{n}x_{-k}^{p}+\frac{\left(a\alpha\right)^{n}-1}{a\alpha-1}\left(a\beta+b\right)x_{0}}{x_{0}}, &a\alpha \neq 1.\\
	\end{array}
	\right.,\label{e6}
\\
u_{2n+1}&
	= \left\{
		\begin{array}{ll}
		\displaystyle\frac{ay_{-k}^{p}+by_{0}+\left(a\beta+b\right)n y_{0}}{y_{0}}, & a\alpha= 1,\\[1.5em]
		\displaystyle\frac{\left(a\alpha\right)^{n}(ay_{-k}^{p}+by_{0})+\frac{\left(a\alpha\right)^{n}-1}{a\alpha-1}\left(a\beta+b\right)y_{0}}{y_{0}}, & a\alpha \neq 1.
		\end{array}
	\right.\label{e7}
\end{align}
and
\begin{align}
v_{2n}&
	= \left\{
	\begin{array}{ll}
		\displaystyle\frac{y_{-k}^{p}+\left(\alpha b+\beta\right)n y_{0}}{y_{0}}, & a\alpha= 1,\\[1.5em]
		\displaystyle\frac{\left(a\alpha\right)^{n}y_{-k}^{p}+\frac{\left(a\alpha\right)^{n}-1}{a\alpha-1}\left(\alpha b+\beta\right)y_{0}}{y_{0}}, & a\alpha \neq 1.\\
	\end{array}
	\right.,\label{e8}
\\
v_{2n+1}&
	= \left\{
		\begin{array}{ll}
		\displaystyle\frac{\alpha x_{-k}^{p}+\beta x_{0}+\left(\alpha b+\beta\right)n x_{0}}{x_{0}}, & a\alpha= 1,\\[1.5em]
		\displaystyle\frac{\left(a\alpha\right)^{n}(\alpha x_{-k}^{p}+\beta x_{0}) +\frac{\left(a\alpha\right)^{n}-1}{a\alpha-1}\left(\alpha b+\beta\right)x_{0}}{x_{0}}, & a\alpha \neq 1.
		\end{array}
	\right.\label{e9}
\end{align}
Now, by rearranging equation \eqref{e2}, we have
\begin{equation*}
	x_{n}=\frac{x_{n-k}^{p}}{u_{n}},\qquad y_{n}=\frac{y_{n-k}^{p}}{v_{n}},\qquad \forall n\in\mathbb{N}_{0}.
\end{equation*}
Replacing $n$ by $kn+r$, for $r=0,1,\ldots,k-1 = \overline{0,k-1}$, we get
\begin{equation*}
	x_{kn+r}=\frac{x_{k(n-1)+r}^{p}}{u_{kn+r}}
	\quad \text{and} \quad
	y_{kn+r}=\frac{y_{k(n-1)+r}^{p}}{v_{kn+r}},\quad
	\text{respectively}.
\end{equation*}
Iterating the right-hand-side (RHS) of the above equations, we get
\begin{align}
	x_{kn+r}=\frac{x_{r-k}^{p^{n+1}}}{\prod_{i=0}^{n}u_{ki+r}^{p^{(n-i)}}},\qquad
	y_{kn+r}=\frac{y_{r-k}^{p^{n+1}}}{\prod_{i=0}^{n}v_{ki+r}^{p^{(n-i)}}},\quad \forall r = \overline{0,k-1}, \ \ n\in\mathbb{N}_0.\label{e10}
\end{align}
We consider two cases: (i) $k$ is even; (ii) $k$ is odd.\\

\underline{CASE 1: $k$ is even}. Suppose $k=2l$, ($l=1,2,\ldots$). 
	Then, from \eqref{e10} and depending on the parity of $r$, we have 
\[
\forall r = \overline{0,l-1}, \ n\in\mathbb{N}_0:\ \left\{ 
	\begin{array}{c}
		x_{2(ln+r)+j}=\displaystyle\frac{x_{2(r-l)+j}^{p^{n+1}}}{\prod_{i=0}^{n}u_{2(li+r)+j}^{p^{(n-i)}}},\\[2em]
		y_{2(ln+r)+j} =\displaystyle\frac{y_{2(r-l)+j}^{p^{n}}}{\prod_{i=0}^{n}v_{2(li+r)+j}^{p^{(n-i)}}}.
	\end{array}\right. j =0 ,1.
\]

\underline{CASE 2: $k$ is odd}. Now, suppose $k=2l + 1$, ($l=0,1,\ldots$).  
	Then, again from \eqref{e10} and depending on the parity of $n$, we get 
	\begin{equation}
\forall r = \overline{0,2l}, \ n\in\mathbb{N}_0: \left\{ 
	\begin{array}{l}
	x_{(2l+1)(2n) + r}=\displaystyle\frac{x_{r-2l-1}^{p^{2n+1}}}{\left(\prod_{i=0}^{n}u_{(2l+1)(2i)+r}^{p^{2(n-i)}}\right)
						\left(\prod_{i=0}^{n-1}u_{(2l+1)(2i+1)+r}^{p^{2(n-i)-1}}\right)}
	\\[2em]
	x_{(2l+1)(2n+1)+r}=\displaystyle\frac{x_{r-2l-1}^{p^{2n+2}}}{\prod_{i=0}^{n}u_{(2l+1)(2i)+r}^{p^{2(n-i)+1}}u_{(2l+1)(2i+1)+r}^{p^{2(n-i)}}}
	\label{e11}
	\end{array}\right.
\end{equation}
and
\begin{equation}
\forall r = \overline{0,2l}, \ n\in\mathbb{N}_0: \left\{ 
	\begin{array}{l}
	y_{(2l+1)(2n) + r} = \displaystyle\frac{y_{r-2l-1}^{p^{2n+1}}}{\left(\prod_{i=0}^{n}v_{(2l+1)(2i)+r}^{p^{2(n-i)}}\right)
						\left(\prod_{i=0}^{n-1}v_{(2l+1)(2i+1)+r}^{p^{2(n-i)-1}}\right)}
	\\[2em]
	y_{(2l+1)(2n+1)+r} = \displaystyle\frac{y_{r-2l-1}^{p^{2n+2}}}{\prod_{i=0}^{n}v_{(2l+1)(2i)+r}^{p^{2(n-i)+1}}v_{(2l+1)(2i+1)+r}^{p^{2(n-i)}}}
	\label{e14}
	\end{array}\right..
\end{equation}

Here we consider two sub-cases:\\

\underline{Subcase 2.1: $l \neq 0$}. 
From \eqref{e11}-\eqref{e14} and depending on the parity of $n$, we get, for all $n\in\mathbb{N}_0$, the following expressions
\begin{align*}
x_{2((2l+1)n+r)}&=\frac{x_{2(r-l)-1}^{p^{2n+1}}}
				{\left(\prod_{i=0}^{n}u_{2((2l+1)i+r)}^{p^{2(n-i)}}\right) \left(\prod_{i=0}^{n-1}u_{2((2l+1)i+l+r)+1}^{p^{2(n-i)-1}}\right)}
,\quad r=\overline{0,l},
\\
x_{2((2l+1)n+r)+1}&=\frac{x_{r-2l-1}^{p^{2n+1}}}
	{\left(\prod_{i=0}^{n}u_{2((2l+1)i+r)+1}^{p^{2(n-i)}}\right) \left(\prod_{i=0}^{n-1}u_{2((2l+1)i+l+r+1)}^{p^{2(n-i)-1}}\right)}
,\quad r=\overline{0,l-1},
\\
x_{2((2l+1)n+l+r)+1}&=\frac{x_{2(r-l)-1}^{p^{2n+2}}}
	{\prod_{i=0}^{n}u_{2((2l+1)i+r)}^{p^{2(n-i)+1}}u_{2((2l+1)i+l+r)+1}^{p^{2(n-i)}}}
,\quad r=\overline{0,l},
\end{align*}
\begin{align*}
x_{2((2l+1)n+l+r+1)}&=\frac{x_{2(r-l)}^{p^{2n+2}}}
	{\prod_{i=0}^{n}u_{2((2l+1)i+r)+1}^{p^{2(n-i)+1}}u_{2((2l+1)i+l+r+1)}^{p^{2(n-i)}}}
,\quad r=\overline{0,l-1}
\end{align*}
and
\begin{align*}
y_{2((2l+1)n+r)}&=\frac{y_{2(r-l)-1}^{p^{2n+1}}}
	{\left(\prod_{i=0}^{n}v_{2((2l+1)i+r)}^{p^{2(n-i)}}\right) \left(\prod_{i=0}^{n-1}v_{2((2l+1)i+l+r)+1}^{p^{2(n-i)-1}}\right)}
,\quad r=\overline{0,l},
\\
y_{2((2l+1)n+r)+1}&=\frac{y_{r-2l-1}^{p^{2n+1}}}
	{\left(\prod_{i=0}^{n}v_{2((2l+1)i+r)+1}^{p^{2(n-i)}}\right) \left(\prod_{i=0}^{n-1}v_{2((2l+1)i+l+r+1)}^{p^{2(n-i)-1}}\right)}
,\quad r=\overline{0,l-1},
\\
y_{2((2l+1)n+l+r)+1}&=\frac{y_{2(r-l)-1}^{p^{2n+2}}}{\prod_{i=0}^{n}v_{2((2l+1)i+r)}^{p^{2(n-i)+1}}v_{2((2l+1)i+l+r)+1}^{p^{2(n-i)}}}
,\quad r=\overline{0,l},
\\
y_{2((2l+1)n+l+r+1)}&=\frac{y_{2(r-l)}^{p^{2n+2}}}{\prod_{i=0}^{n}v_{2((2l+1)i+r)+1}^{p^{2(n-i)+1}}v_{2((2l+1)i+l+r+1)}^{p^{2(n-i)}}}
,\quad r=\overline{0,l-1}.
\end{align*}

\underline{Subcase 2.2: $l = 0$}. 
Using the fact that in this case $r=0$, we get from \eqref{e11}-\eqref{e14}, for all $n \in \mathbb{N}_0$,
\[
	x_{2n} =\frac{x_{-1}^{p^{2n+1}}}{\left(\prod_{i=0}^{n}u_{2i}^{p^{2(n-i)}}\right) \left(\prod_{i=0}^{n-1}u_{2i+1}^{p^{2(n-i)-1}}\right)},
			\qquad x_{2n+1}=\frac{x_{-1}^{p^{2n+2}}}{\prod_{i=0}^{n}u_{2i}^{p^{2(n-i)+1}}u_{2i+1}^{p^{2(n-i)}}}
\]
and
\[
	y_{2n} =\frac{y_{-1}^{p^{2n+1}}}{\left(\prod_{i=0}^{n}v_{2i}^{p^{2(n-i)}}\right) \left(\prod_{i=0}^{n-1}v_{2i+1}^{p^{2(n-i)-1}}\right)},
			\qquad y_{2n+1}=\frac{y_{-1}^{p^{2n+2}}}{\prod_{i=0}^{n}v_{2i}^{p^{2(n-i)+1}}v_{2i+1}^{p^{2(n-i)}}}.
\]
The following theorem summarizes our previous discussion.
\begin{theorem}\label{th2} 
Let $\left\{(x_{n},y_{n})\right\}_{n\geq -k}$ be a solution of System \eqref{T}.
Then, we have the following:
\begin{enumerate}
    \item[{\rm (a)}] If $k=2l$ ($l=1,2,\ldots$), then for all $n \in \mathbb{N}_0$, we have
  	\[
	\forall r = \overline{0,l-1}, \ n\in\mathbb{N}_0:\ \left\{ 
	\begin{array}{c}
		x_{2(ln+r)+j}=\displaystyle\frac{x_{2(r-l)+j}^{p^{n+1}}}{\prod_{i=0}^{n}u_{2(li+r)+j}^{p^{(n-i)}}},\\[2em]
		y_{2(ln+r)+j} =\displaystyle\frac{y_{2(r-l)+j}^{p^{n}}}{\prod_{i=0}^{n}v_{2(li+r)+j}^{p^{(n-i)}}}.
	\end{array}\right. j =0 ,1.
	\]
  \item[{\rm (b)}] If $k=2l + 1$ ($l=1,2,\ldots$), then for all $n \in \mathbb{N}_0$, we have
	\begin{align*}
x_{2((2l+1)n+r)}&=\frac{x_{2(r-l)-1}^{p^{2n+1}}}
				{\left(\prod_{i=0}^{n}u_{2((2l+1)i+r)}^{p^{2(n-i)}}\right) \left(\prod_{i=0}^{n-1}u_{2((2l+1)i+l+r)+1}^{p^{2(n-i)-1}}\right)}
,\quad r=\overline{0,l},
\\
x_{2((2l+1)n+r)+1}&=\frac{x_{r-2l-1}^{p^{2n+1}}}
	{\left(\prod_{i=0}^{n}u_{2((2l+1)i+r)+1}^{p^{2(n-i)}}\right) \left(\prod_{i=0}^{n-1}u_{2((2l+1)i+l+r+1)}^{p^{2(n-i)-1}}\right)}
,\quad r=\overline{0,l-1},
\\
x_{2((2l+1)n+l+r)+1}&=\frac{x_{2(r-l)-1}^{p^{2n+2}}}
	{\prod_{i=0}^{n}u_{2((2l+1)i+r)}^{p^{2(n-i)+1}}u_{2((2l+1)i+l+r)+1}^{p^{2(n-i)}}}
,\quad r=\overline{0,l},
\\
x_{2((2l+1)n+l+r+1)}&=\frac{x_{2(r-l)}^{p^{2n+2}}}
	{\prod_{i=0}^{n}u_{2((2l+1)i+r)+1}^{p^{2(n-i)+1}}u_{2((2l+1)i+l+r+1)}^{p^{2(n-i)}}}
,\quad r=\overline{0,l-1};
\\
y_{2((2l+1)n+r)}&=\frac{y_{2(r-l)-1}^{p^{2n+1}}}
	{\left(\prod_{i=0}^{n}v_{2((2l+1)i+r)}^{p^{2(n-i)}}\right) \left(\prod_{i=0}^{n-1}v_{2((2l+1)i+l+r)+1}^{p^{2(n-i)-1}}\right)}
,\quad r=\overline{0,l},
\\
y_{2((2l+1)n+r)+1}&=\frac{y_{r-2l-1}^{p^{2n+1}}}
	{\left(\prod_{i=0}^{n}v_{2((2l+1)i+r)+1}^{p^{2(n-i)}}\right) \left(\prod_{i=0}^{n-1}v_{2((2l+1)i+l+r+1)}^{p^{2(n-i)-1}}\right)}
,\quad r=\overline{0,l-1},
\\
y_{2((2l+1)n+l+r)+1}&=\frac{y_{2(r-l)-1}^{p^{2n+2}}}{\prod_{i=0}^{n}v_{2((2l+1)i+r)}^{p^{2(n-i)+1}}v_{2((2l+1)i+l+r)+1}^{p^{2(n-i)}}}
,\quad r=\overline{0,l},
\\
y_{2((2l+1)n+l+r+1)}&=\frac{y_{2(r-l)}^{p^{2n+2}}}{\prod_{i=0}^{n}v_{2((2l+1)i+r)+1}^{p^{2(n-i)+1}}v_{2((2l+1)i+l+r+1)}^{p^{2(n-i)}}}
,\quad r=\overline{0,l-1}.
\end{align*}
  \item[{\rm (c)}] If $k=1$, then for all $n \in \mathbb{N}_0$, we have
\[
	x_{2n} =\frac{x_{-1}^{p^{2n+1}}}{\left(\prod_{i=0}^{n}u_{2i}^{p^{2(n-i)}}\right) \left(\prod_{i=0}^{n-1}u_{2i+1}^{p^{2(n-i)-1}}\right)},
			\qquad x_{2n+1}=\frac{x_{-1}^{p^{2n+2}}}{\prod_{i=0}^{n}u_{2i}^{p^{2(n-i)+1}}u_{2i+1}^{p^{2(n-i)}}};
			\]
			\[
	y_{2n} =\frac{y_{-1}^{p^{2n+1}}}{\left(\prod_{i=0}^{n}v_{2i}^{p^{2(n-i)}}\right) \left(\prod_{i=0}^{n-1}v_{2i+1}^{p^{2(n-i)-1}}\right)},
			\qquad y_{2n+1}=\frac{y_{-1}^{p^{2n+2}}}{\prod_{i=0}^{n}v_{2i}^{p^{2(n-i)+1}}v_{2i+1}^{p^{2(n-i)}}}.
\]
\end{enumerate}
The even and odd terms of the sequences $(u_{n})$ and $(v_{n})$ are given by equations \eqref{e6}-\eqref{e9}.
\end{theorem}

\begin{remark}
Particular cases of System \eqref{T} were investigated in \cite{toua1} ($p=1$, $k=2$, $a=\alpha=1$, $b=\pm1$, $\beta=\pm1$).
It is easy to check that the formulae of the solutions in \cite{toua1}, are direct consequences of the formulae obtained in the present work. 
For example if we take $p=1$, $k=2$, $a=\alpha=1$, $b=\beta=1$, then we recover system \eqref{O}. 
In this case $l=1$, $r=0$, so, it follows from Theorem \eqref{th2}, that
\[
x_{2n}=\frac{x_{-2}}{\prod_{i=0}^{n}u_{2i}},\ \ 
	x_{2n+1}=\frac{x_{-1}}{\prod_{i=0}^{n}u_{2i+1}},\ \ 
y_{2n}=\frac{y_{-2}}{\prod_{i=0}^{n}v_{2i}},\ \ 
	y_{2n+1}=\frac{y_{-1}}{\prod_{i=0}^{n}v_{2i+1}}.
\]

Furthermore, if we let $\alpha=a$ and $\beta=b$ and choose initial conditions such that $y_{-i}=x_{-i},i=0,1,\ldots,k$, 
then System \eqref{T} will reduced to the nonlinear difference equation
\begin{equation*}
x_{n+1}=\frac{x_{n-k+1}^{p}x_{n}}{a x_{n-k}^{p}+b x_{n}},
\qquad n \in \mathbb{N}_{0},\ p, k\in \mathbb{N}.
\end{equation*}
\end{remark}

\section{The Case $p=1$}

In this section, we focus our attention on a special case of system \eqref{T}.
In particular, we examine the boundedness, the asymptotic behavior and periodicity of solutions of system \eqref{T} with $p=1$, 
i.e, the system
\begin{equation}
	x_{n+1} = \frac{x_{n-k+1}y_{n}}{ay_{n-k}+by_{n}},\qquad
	y_{n+1} = \frac{y_{n-k+1}x_{n}}{\alpha x_{n-k}+\beta x_{n}},
	\qquad n \in\mathbb{N}_{0},\ k \in \mathbb{N}.\label{e24}
\end{equation}
Throughout this section, we also assume that the set of initial values $\min\{x_{-i}, y_{-i}\}_{i=-k}^0$ satisfy the inequality conditions 
$\min\{x_{-i}, y_{-i}\} \geq 0$ and $\min\{x_{-k} + x_0, y_{-k} + y_0\} > 0$.

We start with the following theorem concerning the boundedness of solutions of system \eqref{e24}.

\begin{theorem}
\label{thm3}
Consider the system \eqref{e24} such that
\begin{enumerate}
\item[{\rm (H.1)}] $\min\{b, \beta\} \geq 1$; or 
\item[{\rm (H.2)}] $\min\{a, \alpha\} \geq 1$, $ay_{-k} \geq y_{0}$ and $\alpha x_{-k}\geq x_{0}$.
\end{enumerate} 
Then, every positive solution is bounded.
\end{theorem}

\begin{proof}
Let $\left\{(x_{n},y_{n})\right\}_{n\geq -k}$ be a solution of \eqref{e24}.

\underline{Hypothesis (H.1) is satisfied}.
Suppose that $\min\{b,\beta\} \geq1$, then it follows from system \eqref{e24} that
\[
	x_{n+1}\leq\frac{x_{n-k+1}}{b}\leq x_{n-k+1}
	\quad \text{and} \quad
	y_{n+1}\leq\frac{y_{n-k+1}}{b}\leq y_{n-k+1}
	\quad \text{for all}\ n \in \mathbb{N}_0. 
\]
and so the subsequences
$\{x_{kn-i}\}_{n \geq 0}$ and $\{y_{kn-i}\}_{n \geq 0}$, $i = 0, \ldots, k-1$, are decreasing. 
Moreover, we have
\[
x_{n}\leq \max_{i = \overline{0,k-1}}\left\{\frac{x_{-i}}{b}\right\}
\quad \text{and} \quad
y_{n}\leq \max_{i = \overline{0,k-1}}\left\{\frac{y_{-i}}{\beta}\right\}
\quad \text{for all}\ n \in \mathbb{N}_0.
\]
Thus, the solution solution is bounded.

\underline{Hypothesis (H.2) is satisfied}.
If, on the other hand,  $\min\{a,\alpha\}\geq1$, $ay_{-k}\geq y_{0}$ and $\alpha x_{-k}\geq x_{0}$, 
then it follows from \eqref{e24} that for $n=0$, we get
\[
x_{1}\leq\frac{x_{-k+1}y_{0}}{ay_{-k}}\leq x_{-k+1} 
\quad \text{and} \quad
y_{1}\leq\frac{y_{-k+1}x_{0}}{\alpha x_{-k}}\leq y_{-k+1}
\]
and from this, together with the assumption that $\min\{a,\alpha\}\geq1$, we get for $n=1$
\[
x_{2}\leq\frac{x_{-k+2}y_{1}}{ay_{-k+1}}\leq x_{-k+2}
\quad \text{and} \quad
y_{2}\leq\frac{y_{-k+2}x_{1}}{\alpha x_{-k+1}}\leq y_{-k+2}.
\]
Continuing the process, we obtain, for $n=k-1$, 
\[
x_{k} \leq \frac{x_{0}y_{k-1}}{ay_{-1}}\leq x_{0}
\quad \text{and} \quad
y_{k} \leq \frac{y_{0}x_{k-1}}{\alpha x_{-1}}\leq y_{0}.
\]
It follows by induction that the subsequences
$\{x_{kn-i}\}_{n \geq 0}$ and $\{y_{kn-i}\}_{n \geq 0}$, $i = 0,\ldots, k-1$, are decreasing. 
Furthermore, we have
\[
x_{n}\leq \max_{i = \overline{0,k-1}}\left\{x_{-i}\right\}
\quad \text{and} \quad
y_{n}\leq \max_{i = \overline{0,k-1}}\left\{y_{-i}\right\}
\quad \text{for all}\ n \in \mathbb{N}_0.
\]
Hence, in this case, the solution solution is also bounded. 
This completes the proof of the theorem.
\end{proof}

In the next theorem, we give the necessary and sufficient conditions for the solutions of system \eqref{e24} to be periodic of period $k$ (not necessary prime).

\begin{theorem} Let $\left\{(x_{n},y_{n})\right\}_{n\geq -k}$ be a solution of \eqref{e24}.
Then, $(x_{n},y_{n})=(x_{n-k},y_{n-k})$ for all $n\in \mathbb{N}_{0}$ if and only if 
$(x_{0},y_{0})=(x_{-k},y_{-k})$ and $a+b=\alpha+\beta=1$.
\end{theorem}
\begin{proof}
First, assume that $(x_{n},y_{n})=(x_{n-k},y_{n-k})$ for all $n\in \mathbb{N}_{0}$.
Particularly, we have $(x_{0},y_{0})=(x_{-k},y_{-k})$ and
\[
x_{-k+1}=x_{1}=\frac{x_{-k+1}y_{0}}{ay_{-k}+by_{0}}
\quad \text{and} \quad
y_{-k+1}=y_{1}=\frac{y_{-k+1}x_{0}}{\alpha x_{-k}+\beta x_{0}}.
\]
These equations imply that 
\[
\frac{1}{a+b}= \frac{1}{\alpha+\beta} = 1
\quad \text{or equivalently}, \quad
a+b=\alpha+\beta=1
\]
Conversely suppose that $(x_{0},y_{0})=(x_{-k},y_{-k})$ and $a+b=\alpha+\beta=1$.
Then, from \eqref{e24} we get
\[
x_{1}=\frac{x_{-k+1}y_{0}}{ay_{-k}+by_{0}}=\frac{x_{-k+1}}{a+b}=x_{-k+1},
\qquad y_{1}=\frac{y_{-k+1}x_{0}}{\alpha x_{-k}+\beta
x_{0}}=\frac{y_{-k+1}}{\alpha+\beta}=y_{-k+1}.
\]
Again, from \eqref{e24}, and using the above relation we get
\[
x_{2}=\frac{x_{-k+2}y_{1}}{ay_{-k+1}+by_{1}}=\frac{x_{-k+2}}{a+b}=x_{-k+2},
\qquad 
y_{2}=\frac{y_{-k+2}x_{1}}{\alpha x_{-k+1}+\beta x_{1}}=\frac{y_{-k+2}}{\alpha+\beta}=y_{-k+2}.
\]
Continuing the process and by principle of induction, we arrive at the desired result.
\end{proof}

The next result provide the limiting properties of solutions of system \eqref{e24}.

\begin{theorem}
Let $\left\{(x_{n},y_{n})\right\}_{n\geq -k}$ be a solution of system \eqref{e24}. 
Then, the following statements hold.

\begin{enumerate}
       \item[{\rm (a)}] If $a\alpha >1$, then $\displaystyle\lim_{n\rightarrow\infty}(x_{n},y_{n})=(0,0)$.
       \item[{\rm (b)}] If $a\alpha =1$, then $\displaystyle\lim_{n\rightarrow\infty}(x_{n},y_{n})=(0,0)$.
       \item[{\rm (c)}] If $a\alpha <1$, then
    	\[
    		\lim_{n\rightarrow \infty} x_{n}=
    		\left\{
    			\begin{array}{ll}
				0, & A > 1,\\[0.5em]
				\infty, & A < 1.
   			 \end{array}
    		\right.\quad \text{and} \quad
  		 \lim_{n\rightarrow \infty} y_{n}=
    		 \left\{
    			\begin{array}{ll}
				0, & B > 1,\\[0.5em]
				\infty, & B < 1.
    		\end{array}
    		\right.,
     \]
     where $A = \dfrac{a\beta+b}{1-a\alpha}$ and $B = \dfrac{\alpha b+\beta}{1-a\alpha}$, respectively.
   \end{enumerate}

\end{theorem}

\begin{proof}
We only prove detailedly properties (a), (b) and (c) for the limits of $x_{n}$. 
The limits of $y_{n}$ follows a similar inductive lines. 
First, note that from \eqref{e10} the limit of $x_{kn+r}$ as $n\rightarrow\infty$
depends on the limit of $u_{kn+r}$ as $n\rightarrow\infty$, which, on the other hand, depends on the value of $a\alpha$.
\begin{enumerate}

    \item[(a)] When $a\alpha>1$, $\dfrac{\left(a\alpha\right)^{n}-1}{a\alpha-1}\rightarrow \infty$ as $n\rightarrow \infty$. 
    			So, from \eqref{e6} and \eqref{e7}, we have $u_{n}\rightarrow \infty$ as $n\rightarrow\infty$.
    			Then, in view of \eqref{e10}, $x_{n}\rightarrow 0$ as $n\rightarrow\infty$. 
    			Similarly, we obtain $y_{n}\rightarrow 0$ as $n\rightarrow\infty$.
    \item[(b)] When $a\alpha=1$, then from \eqref{e6} and \eqref{e7} we get
			$
				\lim_{n\rightarrow\infty}u_{n}
					=\lim_{n\rightarrow\infty}u_{2n}
					=\lim_{n\rightarrow\infty}u_{2n+1}
					=\lim_{n\rightarrow\infty}(a\beta+b)n
					=\infty
			$.
		Hence, from \eqref{e10}, we have $x_{n}\rightarrow 0$ as $n\rightarrow\infty$.
		Similarly, we have $y_{n}\rightarrow 0$ as $n\rightarrow\infty$.
    \item[(c)] When $a\alpha<1$, then $(a\alpha)^{n}\rightarrow 0$ as $n\rightarrow \infty$. 
    			So, in reference to \eqref{e6} and \eqref{e7}, we have
			\begin{align*}
			\lim_{n\rightarrow\infty}u_{n}
				&=\lim_{n\rightarrow\infty}u_{2n}
				=\lim_{n\rightarrow\infty}u_{2n+1}
				=\frac{a\beta+b}{1-a\alpha} 
				=:A,\\
			\lim_{n\rightarrow\infty}v_{n}
				&=\lim_{n\rightarrow\infty}v_{2n}
				=\lim_{n\rightarrow\infty}v_{2n+1}
				=\frac{\alpha b+\beta}{1-a\alpha}
				=:B.
			\end{align*}
		Let $r \in \{0,\ldots, k-1\}$ be fixed.
		If $A > 1$, then $\lim_{n\rightarrow\infty}\prod_{m=0}^{n} u_{km+r}=\infty$. 
    		Therefore, $\lim_{n\rightarrow\infty} x_{kn+r}=0$ or equivalently, $\displaystyle{\lim_{n\rightarrow\infty}}x_{n}=0$.
    		If, on the other hand, $A < 1$, then $\lim_{n\rightarrow\infty}\{1/u_{kn+r}\}={\lim_{n\rightarrow\infty}}\{1/u_{n}\}=A>1$.
		Hence, we have the product limit $\lim_{n\rightarrow\infty}\prod_{m=0}^{n}\{1/u_{km+r}\}=\infty$. 
		Thus, $\lim_{n\rightarrow\infty} x_{kn+r}=\infty$ or equivalently, $\lim_{n\rightarrow\infty}x_{n}=\infty$.
		
		Meanwhile, if $B>1$, then $\lim_{n\rightarrow\infty}\prod_{m=0}^{n} v_{km+r}=\infty$ which, in turn, would imply that $\lim_{n\rightarrow\infty} y_{n}=0$.
		However, if $B<1$, then we have $\lim_{n\rightarrow\infty} \{1/v_{kn+r}\}=B>1$. 
		This limit would then yield the product limit $\lim_{n\rightarrow\infty}\prod_{m=0}^{n}\{1/v_{km+r}\}=\infty$.
		Thus, $\lim_{n\rightarrow\infty} y_{kn+r}=\infty$ or equivalently, $\lim_{n\rightarrow\infty} y_{n}=\infty$.
		This proves the third case, completing the proof of the theorem.
\end{enumerate}

\end{proof}

The next theorem provides the behavior of solutions of system \eqref{e24} for the cases $A = 1$ and $B = 1$.

\begin{theorem}
\label{thm6}
Let $\left\{(x_{n},y_{n})\right\}_{n\geq -k}$ be a solution of \eqref{e24} with $k= 2l$ $(l=1,2,\ldots)$. 
Assume that $|a\alpha|<1$. 
Then, the following statements hold.
\begin{enumerate}
    \item[{\rm (a)}] If $A = 1$ (resp. $B = 1$) and $x_{-k}\neq x_{0}$ (resp. $y_{-k}\neq y_{0}$), 
    			then the sub-sequences $\{x_{kn+2r}\}$ (resp. $\{y_{kn+2r}\}$), for all $r=\overline{0,\frac{k}{2}},$ are convergent.
    \item[{\rm (b)}] If $A = 1$ (resp. $B = 1$) and $x_{-k}=x_{0}$ (resp. $y_{-k}=y_{0}$), then $x_{kn+2r}=x_{2r-k}$ (resp. $y_{kn+2r}=y_{2r-k}$) for all $r=\overline{0,\frac{k}{2}}$.
    \item[{\rm (c)}] If $A = 1$ (resp. $B = 1$) and $ay_{-k}\neq (1-b)y_{0}$ (resp. $\alpha x_{-k}\neq (1-\beta)x_{0}$), 
    			then the sub-sequences $\{x_{kn+2r+1}\}$ (resp. $\{y_{kn+2r+1}\}$), for all $r=\overline{0,\frac{k}{2}}$, are convergent.
    \item[{\rm(d)}] If $A = 1$ (resp. $B = 1$) and $ay_{-k}=(1-b)y_{0}$ (resp. $\alpha x_{-k}=(1-\beta)x_{0}$), 
    			then $x_{kn+2r+1}=x_{2r-k+1}$ (resp. $y_{kn+2r+1}=y_{2r-k+1}$) for all $r=\overline{0,\frac{k}{2}}$.
    \end{enumerate}

\end{theorem}
\begin{proof}
We only prove the results for the sub-sequences $\{x_{kn+2r+i}\}$ ($i=0,1$).
The same lines of proof, however, can be followed inductively to prove the results for the sub-sequences $\{y_{kn+2r+i}\}$ ($i=0,1$).

First, we note that in all cases
\[
u_{2n}=\frac{(a\alpha)^{n}(x_{-k}-x_{0})}{x_{0}}+1
\quad \text{and} \quad
u_{2n+1}=\frac{(a\alpha)^{n}[ay_{-k}+(b-1)y_{0}]}{y_{0}}+1.
\]
Since $ \lim_{n\rightarrow\infty}(a\alpha)^{n}=0$, then there exists $n_0\in \mathbb{N}$ such that $\forall n\geq n_0$ for all $c\in \mathbb{R}$, we have $|c(a\alpha)^{n}|<1$.

\begin{enumerate}
\item[(a)] By Theorem \eqref{th2}, we have
\begin{align*}
	x_{2(ln+r)}
		&=\frac{\displaystyle{x_{2r-2l}}}{\displaystyle{\prod_{i=0}^{n}\left(1+\frac{(a\alpha)^{li+r}(x_{-k}-x_{0})}{x_{0}}\right)}}\\
		&=\frac{\displaystyle{x_{2r-2l}}}{\displaystyle{c_{1}(n_{0})\exp\sum_{i=n_{0}}^{n}\ln\left[1+\frac{(a\alpha)^{li+r}(x_{-k}-x_{0})}{x_{0}}\right]}}.
\end{align*}
Using a propriety of logarithms, we have
	\[
		\ln\left[1+\frac{(a\alpha)^{n}(x_{-k}-x_{0})}{x_{0}}\right]\sim \frac{(a\alpha)^{n}(x_{-k}-x_{0})}{x_{0}}.
	\]
Now, since the term $\sum(a\alpha)^{n}$ is a geometric sum, with $|a\alpha|<1$, then the sum 
\[
\sum_{i\geq n_{0}}\frac{x_{-k}-x_{0}}{x_{0}}(a\alpha)^{li+r}
\] 
is convergent.
The desired result then follows.

\item[(b)] The result is immediate since $u_{2n}=1$ in this case.

\item[(c)] The proof is similar to item (a). 
	       That is, by Theorem \eqref{th2}, we have
		\begin{align*}
			x_{2(ln+r)}
				&=\frac{\displaystyle{x_{2r-2l+1}}}{\displaystyle{\prod_{i=0}^{n}\left[1+\frac{(a\alpha)^{n}(ay_{-k}+(b-1)y_{0})}{y_{0}}\right]}}\\
				&=\frac{\displaystyle{x_{2r-2l+1}}}{\displaystyle{c_{2}(n_{0})\exp\sum_{i=n_{0}}^{n}\ln\left[1+\frac{(a\alpha)^{li+r}(ay_{-k}+(b-1)y_{0})}{y_{0}}\right]}},
		\end{align*}
Again, using a propriety of logarithm, we have
	\[
	\ln\left[1+\frac{(a\alpha)^{n}(ay_{-k}+(b-1)y_{0})}{y_{0}}\right]\sim \frac{(a\alpha)^{n}(ay_{-k}+(b-1)y_{0})}{y_{0}}.
	\]
Since $\sum(a\alpha)^{n}$ is a geometric sum, then the sum 
\[
\sum_{i\geq n_{0}}\frac{ay_{-k}+(b-1)y_{0}}{y_{0}}(a\alpha)^{li+r}
\] 
is convergent. 
Hence, conclusion follows.

\item[(d)] As in item (b), the result is immediate since, in this case, $u_{2n + 1}=1$.
\end{enumerate}
\end{proof}

\section{Numerical Examples}
In this last and final section we provide several numerical examples to illustrate the results we have exhibited in the previous section.
In these examples, the initial values are chosen randomly from the unit interval $(0,1)$.
However, the initial points for sequence $\{x_n\}$ are plotted using squares ($\blacksquare$) while the initial points for sequence $\{y_n\}$ are plotted using triangles ($\blacktriangle$). 
The values for each of the parameters $k, a, b, \alpha$ and $\beta$ are indicated for each plot.
These illustrations corroborate Theorems \ref{thm3} to \ref{thm6}.   

\newpage
\begin{multicols}{2}
\noindent\scalebox{0.45}{\includegraphics{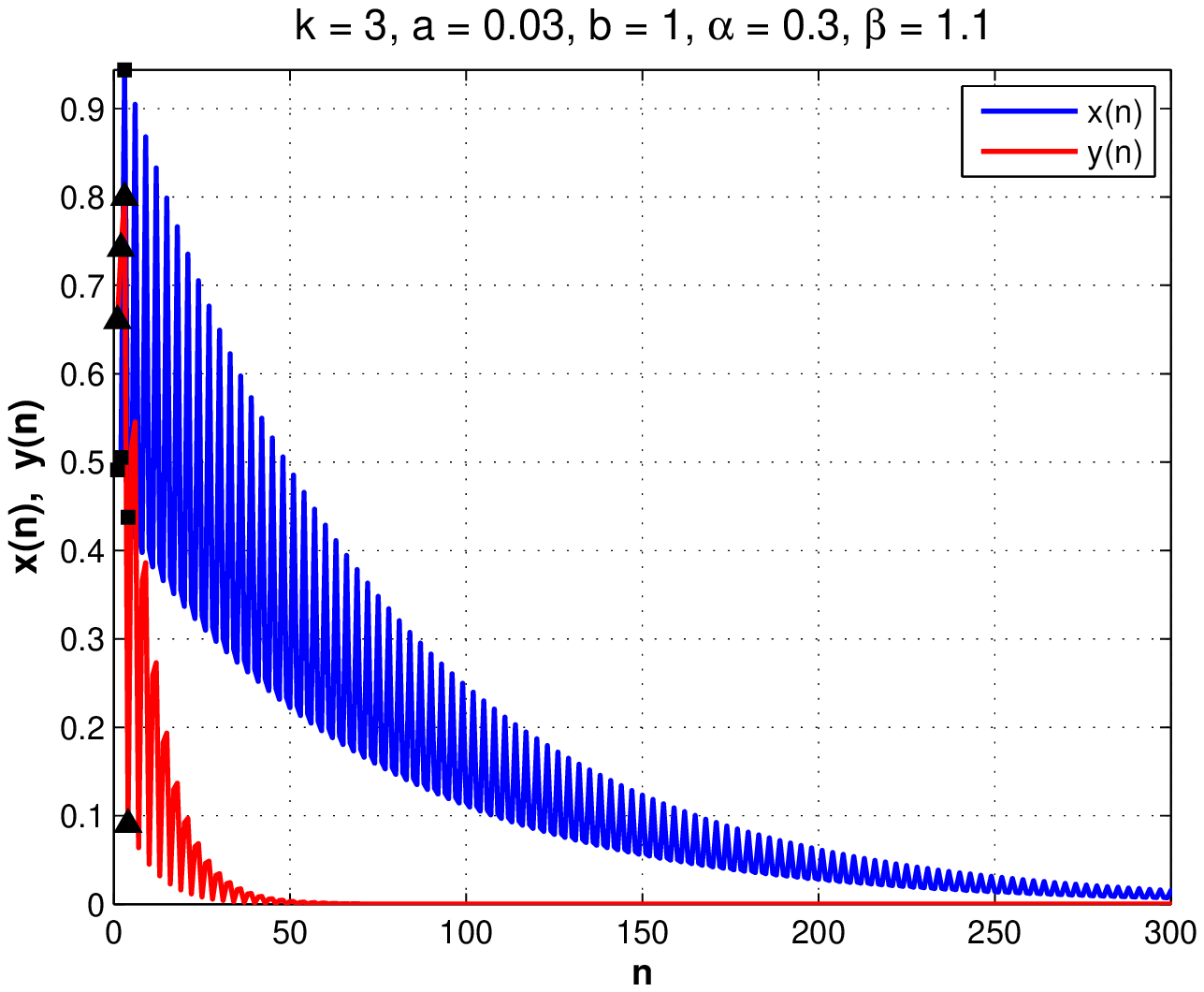}}
\scalebox{0.45}{\includegraphics{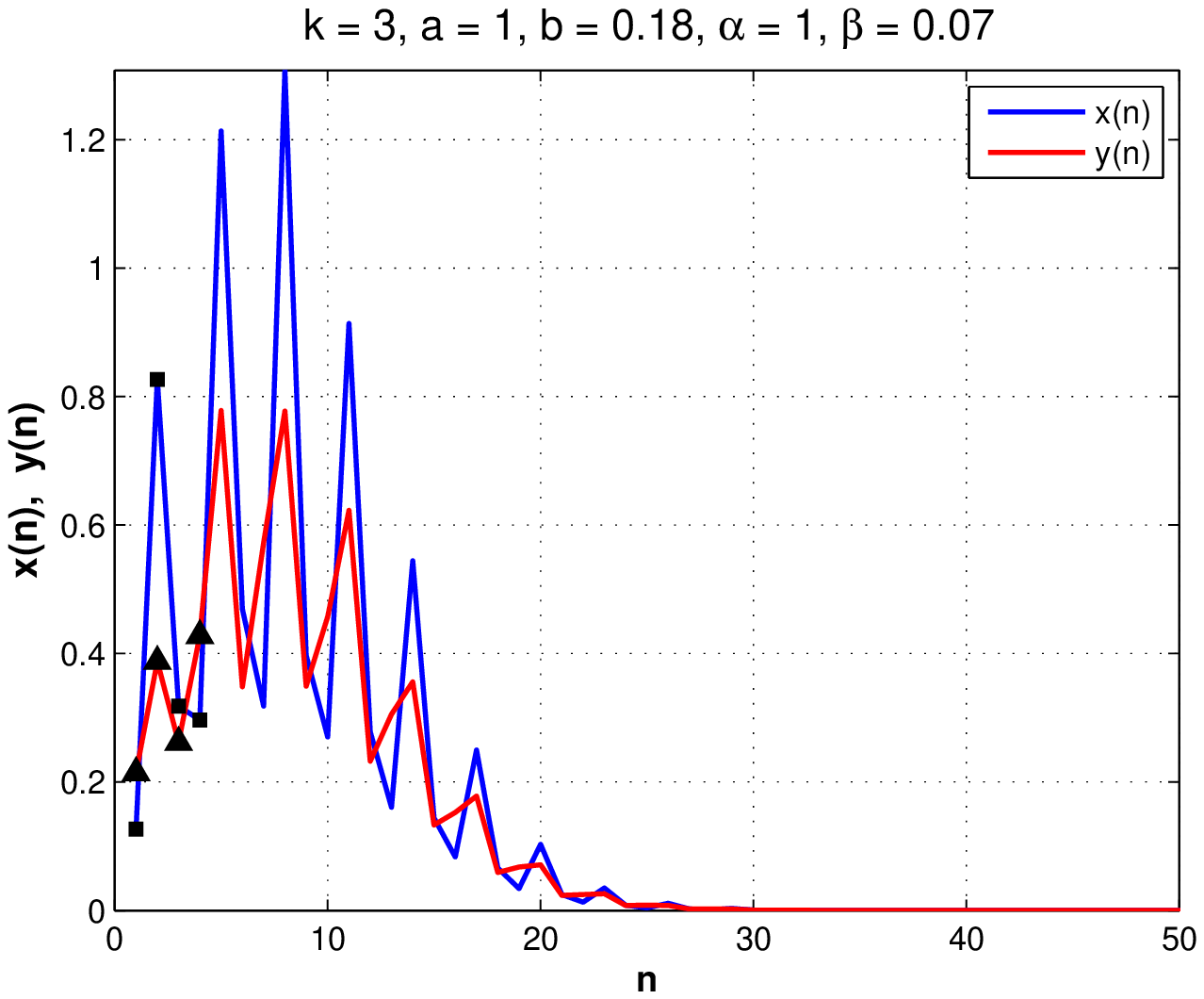}}
\scalebox{0.45}{\includegraphics{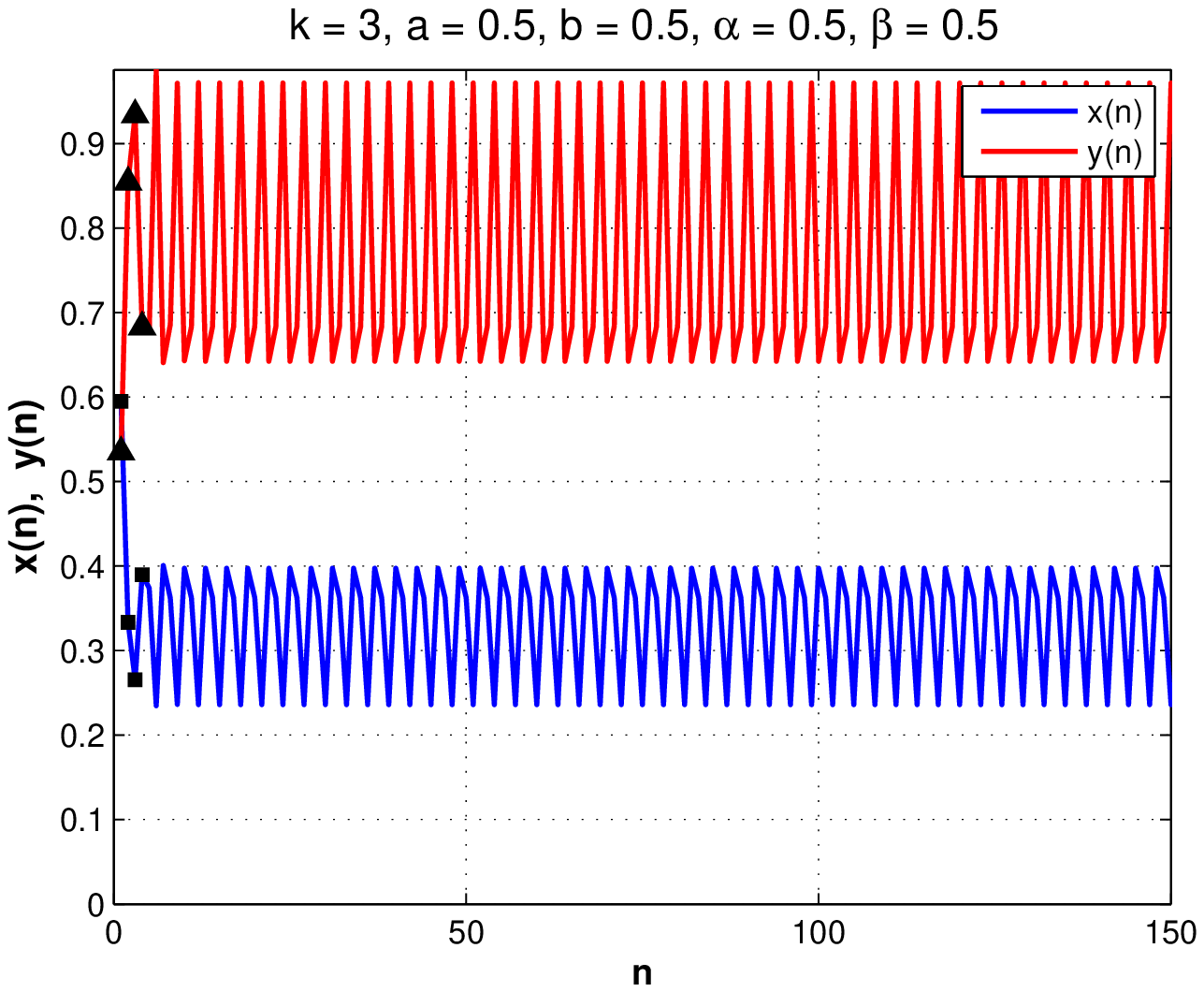}}
\scalebox{0.45}{\includegraphics{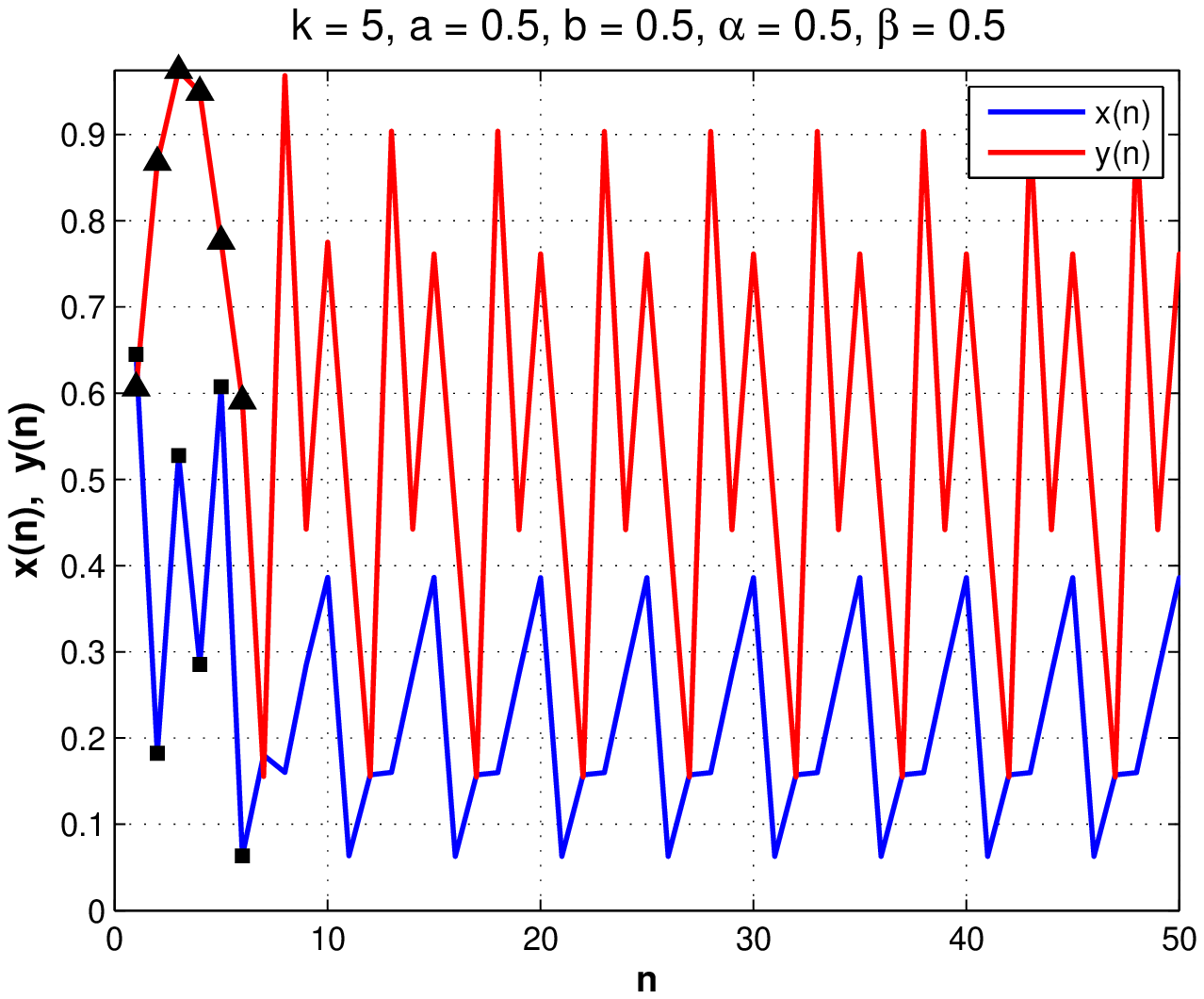}}
\scalebox{0.45}{\includegraphics{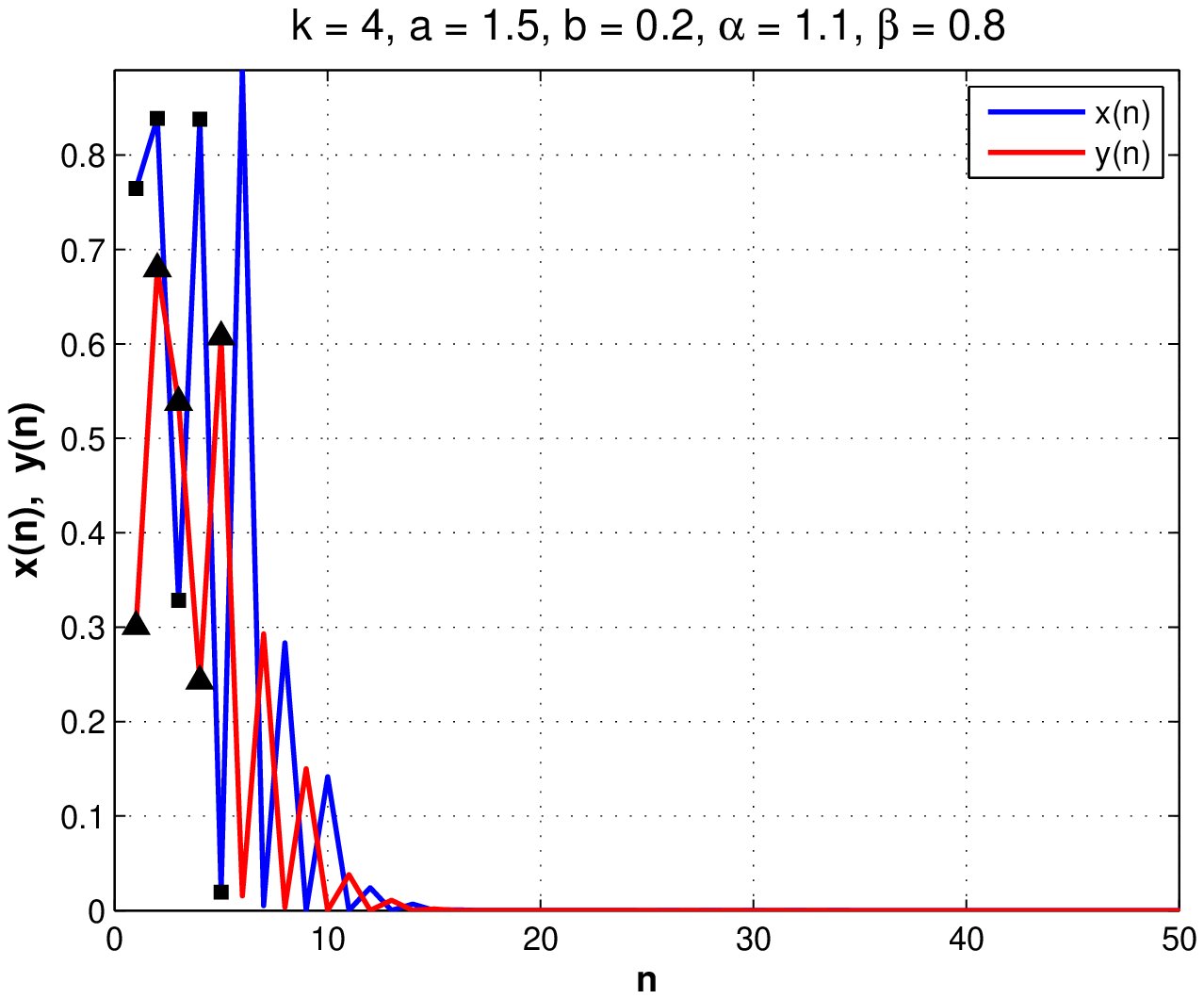}}
\scalebox{0.45}{\includegraphics{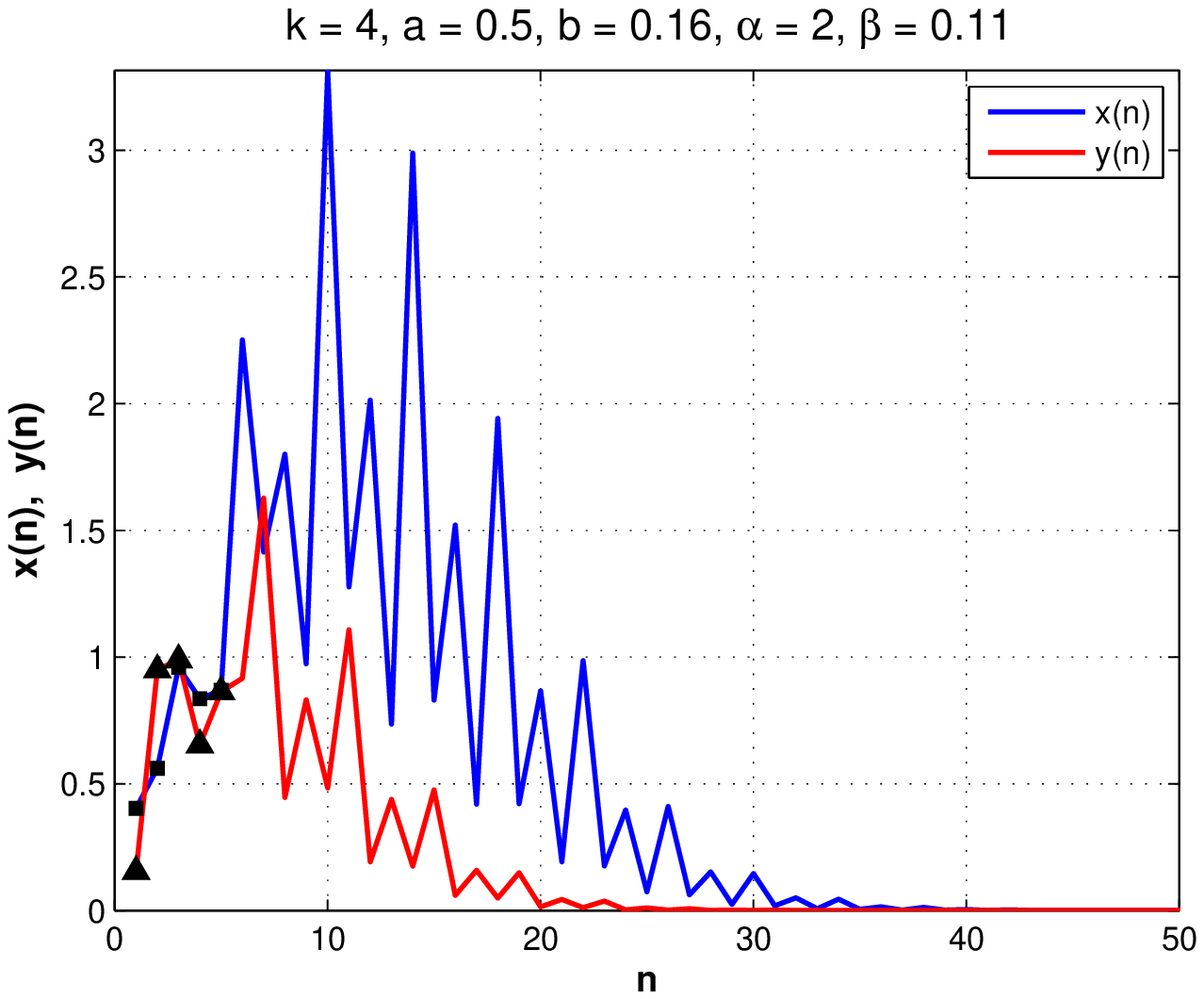}}
\scalebox{0.45}{\includegraphics{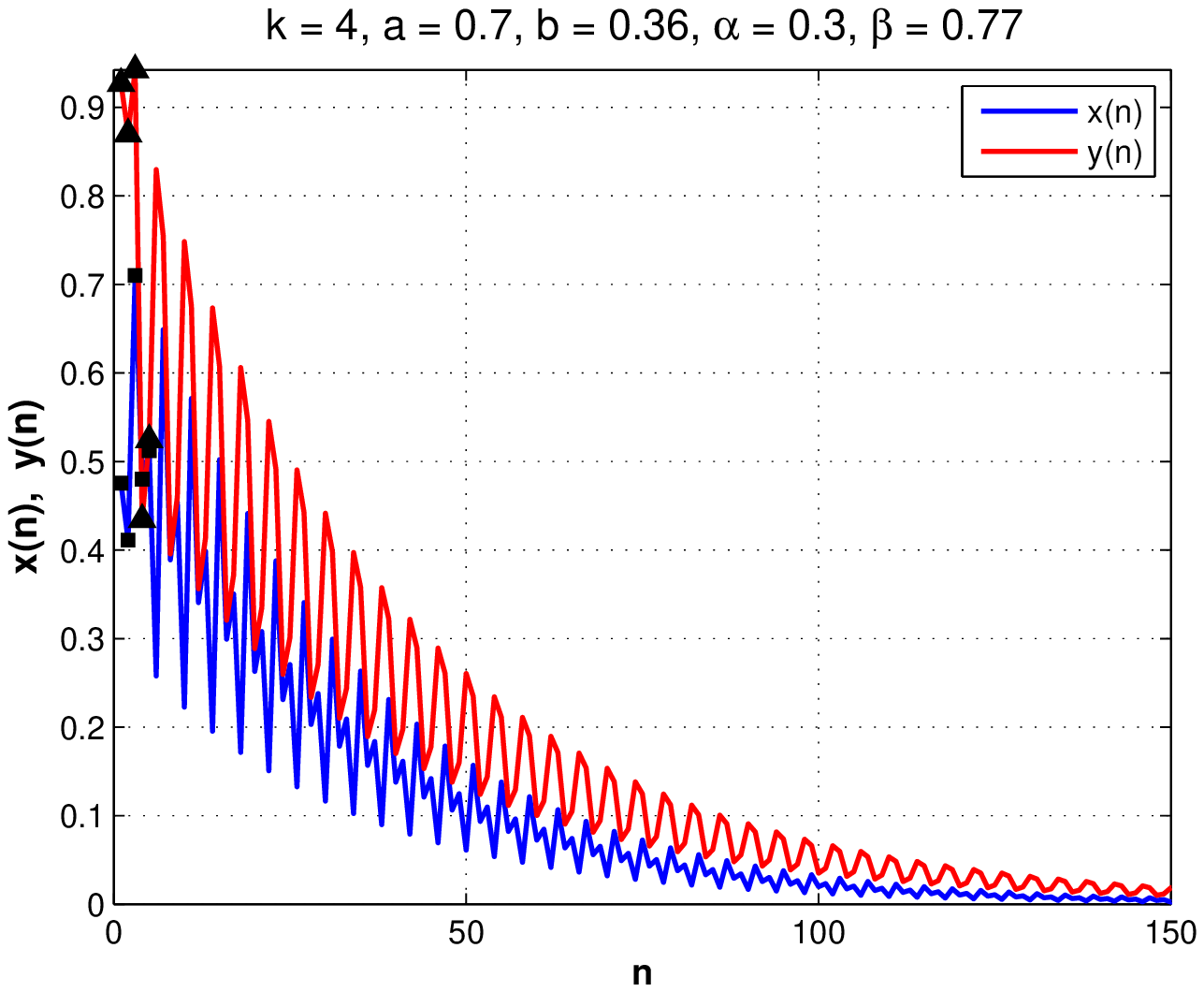}}
\scalebox{0.45}{\includegraphics{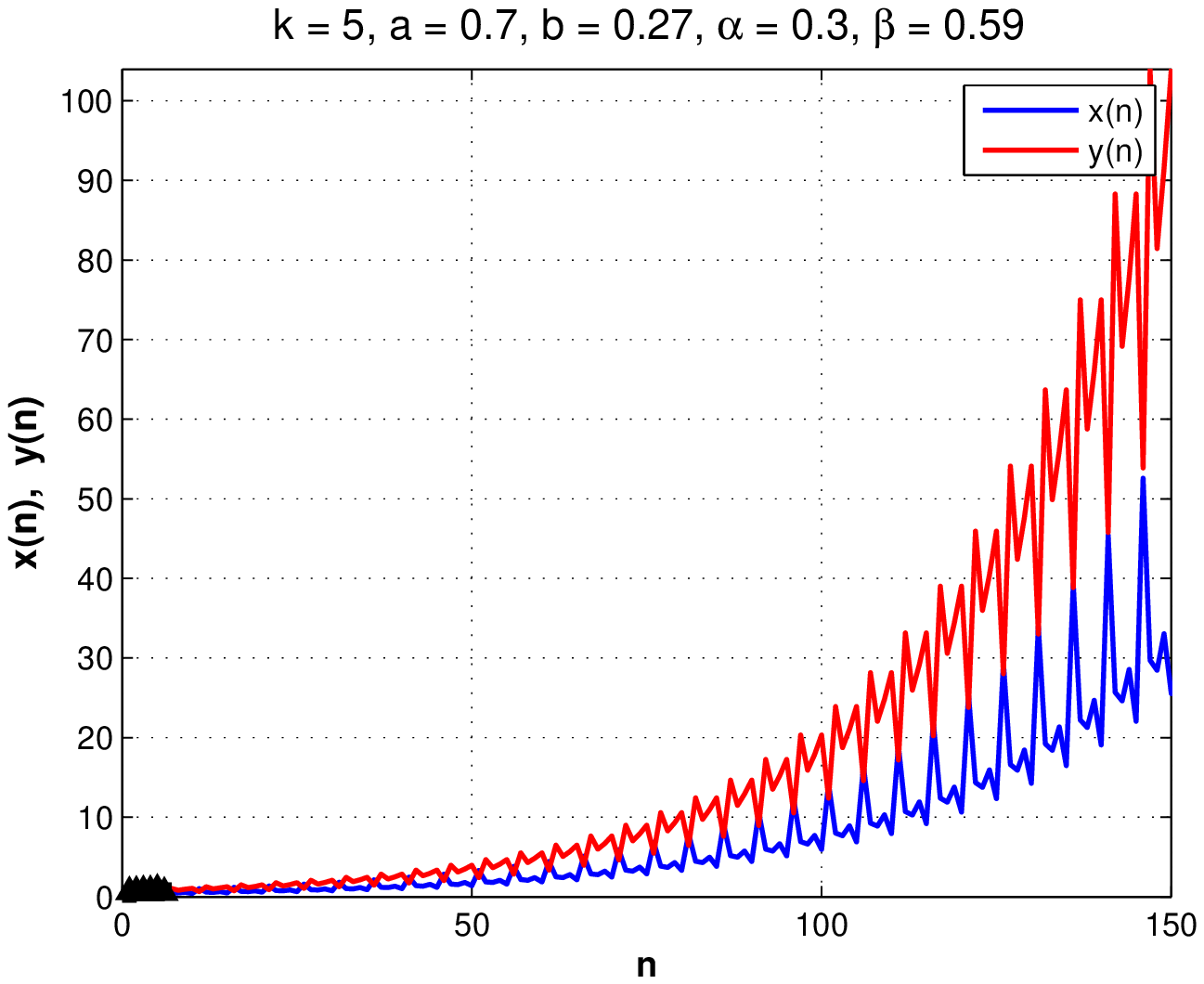}}
\end{multicols}

\end{document}